\documentclass[11pt]{article}
\usepackage{exscale,relsize}
\usepackage{amsmath}
\usepackage{amsfonts}
\usepackage[hidelinks]{hyperref}
\usepackage{amssymb}
\usepackage{calc}
\usepackage{theorem}
\usepackage{pifont}      %needed by dingautolist
\usepackage{array}
\usepackage{color}
\usepackage{graphicx}
\usepackage{comment}
\usepackage[makeroom]{cancel}
 \usepackage{mathpazo}

\newcommand{\scal}[2]{\langle{{#1},{#2}}\rangle}

\newcommand{\RR}{\ensuremath{\mathbb R}}

\newcommand{\RP}{\ensuremath{\RR_+}}
\newcommand{\RM}{\ensuremath{\RR_-}}

\newcommand{\RX}{\ensuremath{\,\left]-\infty,+\infty\right]}}

\newcommand{\menge}[2]{\big\{{#1} \mid {#2}\big\}}

\newcommand{\gra}{\ensuremath{\operatorname{gra}}}

\newcommand{\spa}{\ensuremath{\operatorname{span}}}
\newcommand{\sym}{\ensuremath{\operatorname{sym}}}

\newcommand{\Id}{\ensuremath{\operatorname{Id}}}

\newcommand{\pinf}{\ensuremath{+\infty}}

\renewcommand{\phi}{\ensuremath{\varphi}}
\newtheorem{theorem}{Theorem}[section]
\newtheorem{lemma}[theorem]{Lemma}
\newtheorem{fact}[theorem]{Fact}

\newtheorem{definition}[theorem]{Definition}

\theoremstyle{plain}{\theorembodyfont{\rmfamily}
}
\theoremstyle{plain}{\theorembodyfont{\rmfamily}
}
\theoremstyle{plain}{\theorembodyfont{\rmfamily}
}
\theoremstyle{plain}{\theorembodyfont{\rmfamily}
\newtheorem{example}[theorem]{Example}}
\theoremstyle{plain}{\theorembodyfont{\rmfamily}
\newtheorem{remark}[theorem]{Remark}}
\theoremstyle{plain}{\theorembodyfont{\rmfamily}
}

\newcommand{\pluss}{{\hskip1pt \raise1pt\vbox{\hrule width6pt \vskip1pt
\hrule width6pt}\kern-4pt{\lower1pt\hbox{\vrule height6pt \kern1pt\vrule
height6pt}}\hskip5pt}}

\begin{document}

\title{\textsc
	A class of multi-marginal $c$-cyclically monotone sets\\
	with explicit $c$-splitting potentials}

\author{
	Sedi Bartz\thanks{
		Mathematics, University
		of British Columbia,
		Kelowna, B.C.\ V1V~1V7, Canada. E-mail:
		\texttt{sedi.bartz@ubc.ca}.},~~
	Heinz H.\ Bauschke\thanks{
		Mathematics, University
		of British Columbia,
		Kelowna, B.C.\ V1V~1V7, Canada. E-mail:
		\texttt{heinz.bauschke@ubc.ca}.}
	~~and Xianfu Wang\thanks{
		Mathematics, University of British Columbia, Kelowna, B.C.\
		V1V~1V7, Canada.
		E-mail: \texttt{shawn.wang@ubc.ca}.}}

\date{March 3, 2017}

\maketitle

\begin{abstract}
	\noindent
Multi-marginal optimal transport plans are concentrated on
$c$-splitting sets. It is known that, similar to the
two-marginal case, $c$-splitting sets are $c$-cyclically
monotone. Within a suitable framework, the converse implication
was very recently established by Griessler.  However, for
an arbitrary cost $c$, given a multi-marginal $c$-cyclically
monotone set, the question whether there exists an analogous
explicit construction to the one from the two-marginal case
of $c$-splitting potentials is still open.  When the margins
are one-dimensional and the cost belongs to a certain class,
Carlier proved that the two-marginal projections of a
$c$-splitting set are monotone. For arbitrary products of
sets equipped with cost functions which are sums of
two-marginal costs, we show that the two-marginal monotonicity
condition is a sufficient condition which does give rise
to an explicit construction of $c$-splitting potentials.
Our condition is, in principle, easier to verify than the
one of multi-marginal $c$-cyclic monotonicity. Various
examples illustrate our results.
We show that, in general, our condition is
sufficient; however, it is not necessary. On the other hand,
we conclude that when the margins are one-dimensional
equipped with classical cost functions, our condition is a
characterization of $c$-splitting sets and extends classical
convex analysis.
\end{abstract}
{\small
	\noindent
	{\bfseries 2010 Mathematics Subject Classification:}
	{Primary 49K30, 49N15; Secondary 26B25, 47H05, 52A01, 91B68.}
	% 49K30 - Optimal solutions belonging to restricted classes
	% 49N15 - Duality theory
	% 26B25 - Convexity, generalizations
	% 47H05 - Monotone operators and generalizations
	% 52A01 - Axiomatic and generalized convexity
	% 91B68 - Matching models
	
	\noindent {\bfseries Keywords:}
	$c$-convex, $c$-monotone, $c$-splitting functions, $c$-splitting set, cyclically monotone, 
	Monge-Kantorovich, multi-marginal, optimal transport.
}

\section{Introduction}

In the past decade multi-marginal optimal transport has attracted
considerable attention and is now a rapidly growing field of
research. Applications can be found in mathematical finance,
economics, image processing, tomography, statistics, decoupling
of PDEs, mathematical physics and more. Unified and
detailed accounts of multi-marginal optimal transport theory,
and recent developments and applications, can be found in the surveys
\cite{MGN, Pas2} and references therein. Naturally,
considerable effort is being invested into generalizing 
the much better understood and established optimal transport
theory in the two-marginal case. We focus our attention on an
issue of this flavour which underlies the very basic structure of
optimal transport. Let $(X_1,\mu_1),\ldots,(X_N,\mu_N)$ be
Borel probability spaces. We set $X=X_1\times\cdots\times X_N$ and we
denote by $\Pi(X)$ the set of all Borel probability measures $\pi$ on
$X$ such that the marginals of $\pi$ are the $\mu_i$'s. Let
$c:X\to\RR$ be a cost function. A cornerstone of multi-marginal
optimal transport theory is Kellerer's~\cite{Kel} generalization
of the Kantorovich duality theorem to the multi-marginal case
(recent generalizations of Kellerer's duality theorem are
accounted in the surveys mentioned above). Kellerer's duality
theorem asserts that, in a suitable framework,     
\begin{equation}\label{Kellerer}
\min_{\pi\in\Pi(X)} \int_X c(x)d\pi(x)=\max_{\begin{array}{c}
u_i\in L_1({\mu_i}),\\ 
\sum_{1\leq i\leq N}u_i\ \leq\ c	
	\end{array}} \ \ \sum_{1\leq i\leq N}\int_{X_i}u_i(x_i)d\mu_i(x_i).
\end{equation}
Furthermore, if $\pi$ is a solution of the left-hand side
of~\eqref{Kellerer} and $(u_1,\ldots,u_N)$ is a solution of the
right-hand side of~\eqref{Kellerer}, then $\pi$ is concentrated
on the subset $\Gamma$ of $X$ where the equality $c=\sum_{1\leq
i\leq N}u_i$ holds. In recent publications (see~\cite{KP, Gri,
MP}) such subsets $\Gamma$ of $X$ are referred to as $c$-splitting
sets
(see also Definition~\ref{splitting def} below). It was
observed (see, for example, \cite{KS, KP}) that $c$-splitting
sets are, in fact, $c$-cyclically monotone sets in the
multi-marginal sense (see the explicit
Definition~\ref{monotonicity definitions} and Fact~\ref{splitting
implies cyclic} below). Recent studies and applications of multi
marginal $c$-cyclic monotonicity and related concepts in the
framework of multi-marginal optimal transport include~\cite{GhMa,
GhMo, Pas2, MPC, BG}. In the two-marginal case it is well known
that a subset $\Gamma$ of $X$ is a $c$-splitting set if and only
if $\Gamma$ is $c$-cyclically monotone. Furthermore, given a
$c$-cyclically monotone set $\Gamma$, an explicit construction of
$c$-splitting potentials $(u_1,u_2)$, a generalization of
Rockafellar's explicit construction from classical convex
analysis (see Definition~\ref{Rock antider} and Fact~\ref{Rock}
below) is also well known. In this case $(u_1,u_2)$ is a solution
of the right hand side of~\eqref{Kellerer}. In fact, this
construction holds within the most general framework of
$c$-convexity theory and applies for general sets $X_1,\ X_2$ and
a general coupling (cost) function $c$.  Given additional
properties of $\Gamma$ and $c$, in the two-marginal case, other
explicit techniques, such as integration in $\RR^d$, can
sometimes be applied in order to produce a $c$-splitting solution
$(u_1,u_2)$. In the multi-marginal case $N\geq 3$, for general
sets $X_1\ldots,X_N$ and a general cost function $c:X\to\RR$,
given a $c$-cyclically monotone set $\Gamma\subseteq X$, it is an
interesting open question whether there exist $c$-splitting
potentials $(u_1,\ldots,u_N)$ of $\Gamma$. Very recently, within
a reasonable framework, existence was established by Griessler~\cite{Gri}
using topological arguments. However, even when existence is
known, an explicit construction is still not available and there
is no multi-marginal counterpart of the construction in Definition~\ref{Rock antider}. 

In this paper we focus our attention on a fairly general
and extensively studied class of cost functions $c$
(see~\eqref{our c} below) which are sums of two-marginal coupling
functions. We introduce a class of subsets $\Gamma$ of $X$ by
imposing a $c$-cyclical monotonicity condition on their
two-marginal projections (see~\eqref{Gammaij} below). In the case where $X_i=\RR$ for each $i$, for a certain class of cost functions $c$, Carlier~\cite{Car} established (see also Pass~\cite{Pas2}) that the monotonicity of the two-marginal projections of the set $\Gamma\subseteq X$ is a necessary condition on $\Gamma$ in order that it is a $c$-splitting set. For arbitrary sets $X_i$ equipped with a cost function from the class of our subject matter, we show that, in fact, this condition is sufficient for $\Gamma$ in order to be a $c$-splitting set. This
enables us to employ existing two-marginal solutions in order to
explicitly construct solutions to the multi-marginal problem of
finding $c$-splitting tuples $(u_i,\ldots,u_N)$ for a given set
$\Gamma$ satisfying our condition. Another advantage of our approach is
that, in principle, it is easier to verify that a given set
$\Gamma$ satisfies our condition than verifying that $\Gamma$
satisfies the more general condition of multi-marginal $c$-cyclic monotonicity (see Remark~\ref{easier condition} below). We then focus our attention on classical cost functions (see Definition~\ref{classic costs} below). We provide several examples of our construction and show that our condition on a given set $\Gamma$ is sufficient, however, it is not necessary in order to ensure that $\Gamma$ is $c$-cyclically monotone. More explicitly, we present a $c$-cyclically monotone set $\Gamma$ with an explicit $c$-splitting tuple $(u_1,\ldots,u_N)$ of $\Gamma$ which does not satisfy our condition. On the other hand, when we focus our attention further on classical cost functions with one-dimensional margins, by combining our discussion with the known results regarding the necessity of the two-marginal condition, we conclude a characterization of $c$-splitting sets. Thus, in this case, given any $c$-cyclically monotone set, using our technique, one can construct an explicit $c$-splitting tuple. 

The rest of the paper is organized as follows: In the reminder of
this section we collect necessary notations and conventions. In
Section~2, we discuss multi-marginal $c$-cyclically monotonicity for products of arbitrary sets and present our more particular class of $c$-cyclically monotone sets along with an explicit construction of $c$-splitting tuples. In Section~3 we review classical cost functions and present examples of our construction and of $c$-cyclically monotone sets which do not fit within our framework. Finally, in Section~4 we focus on classical cost functions with one-dimensional marginals and show that in this case our class of sets is precisely the class of $c$-cyclically monotone sets and generalize other characterizations from the two-marginal case.

Let $Y$ and $Z$ be sets. Given a function $f:Y\to\RX$, we say
that $f$ is proper if $f\not\equiv\pinf$. Given a multivalued mapping $M:Y\rightrightarrows Z$, we denote by $\gra(M)$ the graph of $M$, that is, $\gra(M)=\{(y,z)\in Y\times Z\ |\ z\in M(y)\}$. We will denote the identity mapping on a given set by $\Id$. 
Let $S$ be a subset of $Y$. 
The indicator function of $S$ is the function $\iota_S:Y\to\RX$ defined by
$$
\iota_S(y)=\begin{cases}
0,& \text{if $y\in S$;}\\ 
+\infty,&\text{if $y\notin S$.}
\end{cases} 
$$

Throughout our discussion, $N\geq 2$ is a natural number and $I=\{1,\ldots,N\}$ is an index set. 
Unless mentioned otherwise, $X_1,\ldots,X_N$ are arbitrary nonempty sets, $X=X_1\times\cdots\times X_N$ and $c:X\to\RR$ is a function. 
Set $P_i\colon X\to X_i\colon (x_1,\ldots,x_N)\mapsto x_i$ and 
$P_{i,j}\colon X\to X_i\times X_j\colon (x_1,\ldots,x_N)\mapsto
(x_i,x_j)$ for $i$ and $j$ in $\{1,\ldots,N\}$ and when $i<j$. 
Given a subset $\Gamma$ of $X$, 
we set 
\begin{equation}\label{gammai}
\Gamma_i= P_i(\Gamma)
\end{equation}
and 
\begin{equation}\label{Gammaij} 
\Gamma_{i,j}=P_{i,j}(\Gamma).
\end{equation}

Suppose momentarily that $N=2$. Given 
	a function $f_1:X_1\to[-\infty,+\infty]$, its
	$c$-conjugate function $f_1^c:X_2\to[-\infty,+\infty]$ is defined by
	\begin{equation*}
	f_1^c(x_2)=\sup_{x_1\in X_1}\
	\big(c(x_1,x_2)-f_1(x_1)\big),\ \ \ x_2\in X_2.
	\end{equation*}
	Similarly, the $c$-conjugate of 
	$f_2:X_2\to[-\infty,+\infty]$ is the function
	$f_2^c:X_1\to[-\infty,+\infty]$ defined by
	\begin{equation*}
	f_2^c(x_1)=\sup_{x_2\in X_2}\
	\big(c(x_1,x_2)-f_2(x_2)\big),\ \ \ x_1\in X_1.
	\end{equation*}
Clearly, for any function $f:X_1\to\RX$,
\begin{equation}\label{Young-Fenchel}
c(x_1,x_2)\leq f(x_1)+f^c(x_2)\ \ \ \mathrm{for\ all}\ (x_1,x_2)\in X.
\end{equation}
When $N\geq 3$, 
$c$-conjugation is also widely used in the multi-marginal optimal transport literature mentioned above; however, this will not be a part of our discussion. The case of equality in~\eqref{Young-Fenchel}  is captured in the definition of the $c$-subdifferential: Let $f:X_1\to\RX$ be a proper function. The $c$-subdifferential of $f$ is the mapping
	$\partial_c f:X_1\rightrightarrows X_2$ defined by
	\begin{align}
	\partial_c f(x_1)=\ &\big\{x_2\in X_2\ \big|\ f(x_1)+c(x'_1,x_2)\leq f(x'_1)+c(x_1,x_2)\ \ \forall x'_1\in X_1\big\}\nonumber\\
	\nonumber\\
	=\ &\big\{x_2\in X_2\ \big|\ f(x_1)+f^c(x_2)=c(x_1,x_2)\big\}.\label{equality in YF}
	\end{align}
When $M:X_1\rightrightarrows X_2$ and
$\gra(M)\subseteq\gra(\partial_c f)$, we say that $f$ is a
$c$-antiderivative of $M$. In classical settings, say, when
$X_1=X_2=H$ is a real Hilbert space and $c=\scal{\cdot}{\cdot}$
is the inner product on $X$, $f^c=f^*$ is the classical Fenchel
conjugate function of $f$ and $\partial_c f=\partial f$ is the
classical subdifferential of $f$. Let $A:H\to H$ be linear and
bounded. The quadratic form of $A$ is the function $q_A:H\to\RR$
defined by $q_A(x)=\frac{1}{2}\scal{x}{Ax},\ x\in H$. When
$A=\Id$ is the identity on $H$ we will simply write $q = q_{\Id}=\frac{1}{2}\|\cdot\|^2$. Let $n$ be an integer. Then $S_n$ denotes the group of permutations on $n$ elements.

Finally, a remark regarding our conventions is in order. For
convenience, we choose to work with notions, such as $c$-cyclic monotonicity, $c$-splitting set etc., which are compatible with classical two-marginal convex analysis. However, these conventions are not compatible with minimizing the total cost of transportation but, rather, with maximizing it. To make our discussion compatible with optimal transport theory some standard modifications are needed. For example, to make optimal transport compatible with our discussion, one should exchange min for max in the left-hand side of~\eqref{Kellerer}, exchange the max for min in the right-hand side of~\eqref{Kellerer} and, finally, exchange the constraint $\sum_i u_i\leq c$ in the right-hand side of~\eqref{Kellerer} with the constraint $c\leq \sum_i u_i$.

\section{Multi-marginal c-cyclically monotone sets and sets with $c$-cyclically 2-marginal projections}

We begin this section by recalling the notions of $c$-cyclically
monotone sets, the notion of $c$-splitting sets and the relations
between them for general cost functions $c$ in the multi-marginal case.

\begin{definition}\label{monotonicity definitions}
The subset $\Gamma$ of $X$ is said to be $c$-cyclically monotone
of order $n$, $n$-$c$-monotone for short, if for all $n$ tuples
$(x^1_1,\dots,x_N^1),\dots,(x_1^n,\dots,x_N^n)$ in $\Gamma$ and
every $N$ permutations $\sigma_1,\dots,\sigma_N$ in $S_n$,
\begin{equation}\label{cycmondef}
\sum_{j=1}^nc(x_1^{\sigma_1(j)},\dots,x_N^{\sigma_N(j)})\leq \sum_{j=1}^nc(x_1^{j},\dots,x_N^{j});
\end{equation}
$\Gamma$ is said to $c$-cyclically monotone if it is
$n$-$c$-monotone for every $n\in\{2,3,\dots\}$; and 
$\Gamma$ is said to be $c$-monotone if it is $2$-$c$-monotone. 
\end{definition}

\begin{definition}\label{splitting def}
The subset $\Gamma$ of $X$ is said to be a $c$-splitting set if for each $i\in I$ there exists a function $u_i:X_i\to\RX$ such that
\begin{equation*}
c(x_1,\ldots,x_N)\leq\sum_{i=1}^N u_i(x_i)\ \ \ \ \ \ \ \ \ \ \ \ \ \ \ \forall(x_1,\ldots,x_N)\in X
\end{equation*}
and 
\begin{equation*}
c(x_1,\ldots,x_N)=\sum_{i=1}^N u_i(x_i)\ \ \ \ \ \ \ \ \ \ \ \ \ \ \ \forall(x_1,\ldots,x_N)\in \Gamma.
\end{equation*}
In this case we say that $(u_1,\ldots,u_N)$ is a $c$-splitting tuple of $\Gamma$.
\end{definition}

In the case $N=2$ it is well known that $c$-splitting sets are
$c$-cyclically monotone. It was observed that this fact holds for
any $N\geq 2$ as well (see, for example, \cite{KS, KP,Gri}). For
completeness of our discussion and for the reader's convenience 
we include a proof of this fact.  

\begin{fact}\label{splitting implies cyclic}
Let $\Gamma$ be a $c$-splitting subset of $X$. 
Then $\Gamma$ is a $c$-cyclically monotone set.
\end{fact}

\begin{proof}
Let $(u_1\,\ldots,u_N)$ be a $c$-splitting tuple of $\Gamma$, let $(x^1_1,\dots,x_N^1),\dots,(x_1^n,\dots,x_N^n)$ be points in $\Gamma$ and let $\sigma_1,\dots,\sigma_N$ be permutations in $S_n$. Then
\begin{align*}
&\sum_{j=1}^nc(x_1^{\sigma_1(j)},\dots,x_N^{\sigma_N(j)})\leq\sum_{j=1}^n\sum_{i=1}^N u_i(x_i^{\sigma_i(j)})=\sum_{i=1}^N\sum_{j=1}^n u_i(x_i^{\sigma_i(j)})\\ =&\sum_{i=1}^N\sum_{j=1}^n u_i(x_i^{j})
=\sum_{j=1}^n\sum_{i=1}^N u_i(x_i^{j})=\sum_{j=1}^nc(x_1^{j},\dots,x_N^{j}),
\end{align*}
as required.
\end{proof}

\begin{remark}\label{trivial monotonicity}
The following known and elementary facts follow immediately: If $h:X\to\RR$ is a separable function, then there is equality in the definition of $h$-cyclic monotonicity on all of $X$. More explicitly, if $h(x_1,\ldots,x_N)=h_1(x_1)+\cdots+h_N(x_N)$, then for any $n$ tuples $(x^1_1,\dots,x_N^1),\dots,(x_1^n,\dots,x_N^n)$ in $X$ and any $N$ permutations $\sigma_1,\dots,\sigma_N$ in $S_n$,
\begin{align*}
&\sum_{j=1}^nh(x_1^{\sigma_1(j)},\dots,x_N^{\sigma_N(j)})=\sum_{j=1}^nh_1(x_1^{\sigma_1(j)})+\cdots+\sum_{j=1}^nh_N(x_N^{\sigma_N(j)})\\
=& \sum_{j=1}^nh_1(x_1^{j})+\cdots+\sum_{j=1}^nh_N(x_N^{j})=\sum_{j=1}^nh(x_1^{j},\dots,x_N^{j}).
\end{align*}
Consequently, a subset $\Gamma$ of $X$ is $c$-cyclically
monotone if and only if it is $(c+h)$-cyclically monotone for any
separable function $h:X\to\RR$. Furthermore, $(u_1,\ldots,u_n)$
is a $c$-splitting tuple of $\Gamma$ if and only if
$(u_1+h_1,\ldots,u_N+h_N)$ is a $(c+h)$-splitting tuple of
$\Gamma$. Because of the marginal condition plans $\pi\in\Pi$ must satisfy, $\pi$ is an optimal plan for the optimal transport problem with cost $c$ if and only if $\pi$ is optimal for the problem with cost $c+h$.
\end{remark}

In the case $N=2$, the converse implication to the one in
Fact~\ref{splitting implies cyclic} is well known, i.e., 
a subset $\Gamma$ of $X$ is $c$-cyclically monotone if and only if $\Gamma$ is a $c$-splitting set. 
Indeed, this follows from the following generalization of Rockafellar's explicit construction \cite{Rockafellar} from classical convex analysis: 

\begin{definition}\label{Rock antider}
	Suppose that $N=2$, let $\Gamma$ be a nonempty subset of
	$X$, and let $s_1\in\Gamma_1$. 
	With the function $c$, the set $\Gamma$ and the point $s_1$, we associate the function $R_{[c,\Gamma,s_1]}:X_1\to\RX$, defined by
	\begin{equation}
	R_{[c,\Gamma,s_1]}(x_1)=
	\sup_{\begin{array}{c}
		n\in\mathbb{N},\\[3pt]
		x_1^1=s_1,\ x_1^{n+1}=x_1,\\[3pt]
		\big\{(x_1^j,x_2^j)\big\}_{j=1}^n\subseteq\Gamma
		\end{array} }
	\ \ \sum_{j=1}^n \big(c(x_1^{j+1},x_2^j)-c(x_1^j,x_2^j)\big).
	\end{equation}
\end{definition}

\begin{fact}\label{Rock}
	Suppose that $N=2$, let $\Gamma$ be a nonempty
subset of $X$, and let $M:X_1\rightrightarrows X_2$ be the mapping defined via $\gra(M)=\Gamma$. Then $\Gamma$ is $c$-cyclically monotone if and only if $M$
	has a proper $c$-antiderivative. In this case, for any $s_1\in\Gamma_1$,
	the function $R_{[c,\Gamma,s_1]}$ is a proper (and $c$-convex) $c$-antiderivative
	of $M$ which satisfies $R_{[c,\Gamma,s_1]}(s_1)=0$. In fact, $R_{[c,\Gamma,s_1]}$ is proper if and only if $\Gamma$ is $c$-cyclically monotone.
\end{fact}
Thus, given a $c$-cyclically monotone subset $\Gamma$ of $X$, by combining Fact~\ref{Rock} with~\eqref{Young-Fenchel} and~\eqref{equality in YF}, we conclude that $(u_1,u_2)=(R_{[c,\Gamma,s_1]},R^c_{[c,\Gamma,s_1]})$ is a $c$-splitting tuple of $\Gamma$. 

Even though many authors in the optimal transport literature attribute the above generalization (Fact~\ref{Rock}) of Rockafellar's characterization of cyclic monotonicity to the generality of $c$-convexity theory to~\cite{Rus}, it is known by now that such generalized constructions were available outside classical convex analysis and within the context of optimal transport a decade earlier, independently, in~\cite{Bre} and in~\cite{Rochet}. 

Finer properties of $R_{[c,\Gamma,s_1]}$ were studied and employed in order to construct constrained optimal $c$-antiderivatives in~\cite{BR1} and in the context of optimal transport in~\cite{BR2}. 

In the case when $N\geq 3$, even though existence of a $c$-splitting tuple
for a given $c$-cyclically monotone set is now known in a fairly
general framework (see~\cite{Gri}), an analogous construction to
the one of $R_{[c,\Gamma,s_1]}$ for an arbitrary cost function
$c$ on arbitrary sets $X_i$ is currently unavailable. We now
study a class of cases which allows us to apply two-marginal
$c$-splitting tuples in order to construct multi-marginal ones. To this end, from now on we focus our attention on the class of cost functions $c$ of the following form: Suppose that for each $1\leq i<j\leq N$ we are given a two-marginal cost function (or coupling function) $c_{i,j}:X_i\times X_j\to\RR$. Then we study the cost function $c:X\to\RR$ defined by
\begin{equation}\label{our c}
c(x_1,\ldots,x_N)=\sum_{1\leq i<j\leq N}c_{i,j}(x_i,x_j).
\end{equation}
In our main abstract result (Theorem~\ref{main abstract result}
below), we impose a $c_{i,j}$-cyclic monotonicity condition on the
$\Gamma_{i,j}$'s. By doing so, we may use solutions from the
two-marginal case in the multi-marginal case.

 \begin{theorem}\label{main abstract result}
 Let $\Gamma$ be a nonempty subset of $X$. Suppose that for each
 $1\leq i< j\leq N$ the set $\Gamma_{i,j}$ is $c_{i,j}$-cyclically
 monotone. Then $\Gamma$ is $c$-cyclically monotone. Furthermore,
 there exist functions $f_{i,j}:X_i\to\RX$ such that $\Gamma_{i,j}\subseteq
 \gra(\partial_{c_{i,j}}f_{i,j})$ for each $1\leq i< j\leq N$. In
 particular, given $(s_1,\ldots,s_N)\in{\Gamma}$, one can take
 $f_{i,j}=R_{[c_{i,j},\Gamma_{i,j},s_{i}]}$. Consequently, the
 functions $u_i:X_i\to\RX$ defined by
 \begin{equation}\label{c-spliting u}
 u_i(x_i)=\sum_{i<k\leq N}f_{i,k}(x_i)+\sum_{1\leq k<i}f_{k,i}^{c_{k,i}}(x_i)
 \end{equation}
 form a $c$-splitting tuple of $\Gamma$, that is 
\begin{equation}\label{c-splitting thm} 
c (x_1,\ldots,x_N)\leq \sum_{i=1}^N u(x_i)\ \ \ \ \ \ \ \ \ \ \ \ \ \forall(x_1,\ldots,x_N)\in X, 
\end{equation}
and equality in~\eqref{c-splitting thm} holds for every $(x_1,\ldots,x_N)\in\Gamma$. Furthermore, if 
\begin{equation}\label{main furthermore}
\Gamma_{i,j}=\gra(\partial_{c_{i,j}} f_{i,j})\ \ \ \ \text{for
each}\ \ 1\leq i< j\leq N\ \ \ \ \ \ \ \text{and}\ \ \ \ \ \ \ \
\Gamma=\bigcap_{i<j}P_{i,j}^{-1}(\Gamma_{i,j}), 
\end{equation}
then equality in~\eqref{c-splitting thm} holds if and only if $(x_1,\ldots,x_N)\in\Gamma$. 
 \end{theorem}

\begin{proof}
Let $(x_1,\ldots,x_N)\in X$. By applying~\eqref{Young-Fenchel} for each $1\leq i<j\leq N$ we see that
\begin{equation}\label{2-mar Y-F}
c_{i,j}(x_i,x_j)\leq f_{i,j}(x_i)+f_{i,j}^{c_{i,j}}(x_j).
\end{equation}
Summing up, we arrive at
\begin{align}
c(x_1,\ldots,x_N)&=\sum_{1\leq i\leq N}\sum_{i<j\leq N}c_{i,j}(x_i,x_j)\leq \sum_{1\leq i\leq N}\sum_{i<j\leq N} f_{i,j}(x_i)+f_{i,j}^{c_{i,j}}(x_j)\label{tuple inequality}\\
&=\sum_{1\leq i\leq N}\Bigg(\sum_{i<k\leq N}f_{i,k}(x_i)+\sum_{1\leq k<i}f_{k,i}^{c_{k,i}}(x_i)\Bigg)=\sum_{1\leq i\leq N} u_i(x_i).\nonumber
\end{align}
Furthermore, for $(x_1,\ldots,x_N)\in \bigcap_{i<j}P_{i,j}^{-1}(\Gamma_{i,j})$, for each $1\leq i<j\leq N$, since $(x_i,x_j)\in\Gamma_{i,j}\subseteq\gra\partial_c f_{i,j}$, we have equality in \eqref{2-mar Y-F}, which, in turn, implies equality in \eqref{tuple inequality}. Thus, since $\Gamma\subseteq\bigcap_{i<j}P_{i,j}^{-1}(\Gamma_{i,j})$, we see that $(u_1,\ldots,u_N)$ is a $c$-splitting tuple of $\Gamma$.  As a consequence, $c$-cyclic monotonicity of $\Gamma$ now follows from Fact~\ref{splitting implies cyclic}. Finally, if  $\Gamma_{i,j}=\gra(\partial_{c_{i,j}} f_{i,j})$ for each $1\leq i< j\leq N$ and $\Gamma=\bigcap_{i<j}P_{i,j}^{-1}(\Gamma_{i,j})$, then by applying~\eqref{equality in YF}, we see that there is equality in~\eqref{2-mar Y-F} if and only if $(x_i,x_j)\in\Gamma_{i,j}$. Consequently, we see that there is equality in~\eqref{tuple inequality} if and only if $(x_1,\ldots,x_N)\in\bigcap_{i<j}P_{i,j}^{-1}(\Gamma_{i,j})=\Gamma$, which completes the proof.
\end{proof}

We end this section by making the following observation regarding the applicability of Theorem~\ref{main abstract result}.

\begin{remark}\label{easier condition}
In the next section we shall see that the class of sets $\Gamma$ with $c_{i,j}$-cyclically monotone $\Gamma_{i,j}$'s is, in general, a proper subset of the class of $c$-cyclically monotone sets. Nevertheless, we now claim that verifying the $c_{i,j}$-cyclic monotonicity of the $\Gamma_{i,j}$'s is, in principle, a simpler verification than the one of the more general condition of $c$-cyclic monotonicity of $\Gamma$. Indeed, in the case $N=2$, according to Definition~\ref{monotonicity definitions}, we need to check that given any $n$ points $(x_1^1,x_2^1),\ldots,(x_1^n,x_2^n)$ in $\Gamma$ and any permutations $\sigma_1$ and $\sigma_2$ in $S_n$,
\begin{equation}\label{2-mar cyclic}
\sum_{1\leq j\leq n} c(x_1^{\sigma_1(j)},x_2^{\sigma_2(j)})\leq\sum_{1\leq j\leq n} c(x_1^j,x_2^j).
\end{equation}
However, it is well known (see, for example,~\cite{Vil}) that in the case $N=2$, verifying~
\eqref{2-mar cyclic} for any $\sigma_1$ and $\sigma_2$ in $S_n$
is equivalent to verifying~\eqref{2-mar cyclic} for the specific
permutations $\sigma_1=\Id$ and $\sigma_2(j)=(j+1)\mod n,\ 1\leq
j\leq n$. For $N\geq 3$, running this verification for all the
$\Gamma_{i,j}$'s, is, in principle, simpler than running the
verification over all $\sigma_1,\ldots,\sigma_N\in S_n$ (or,
equivalently, over all $\sigma_1,\ldots,\sigma_N\in S_n$ with $\sigma_1=\Id$) in order to verify the $c$-cyclic monotonicity of $\Gamma$.  Furthermore, as we shall see in our examples, for specific cost functions we sometimes have even simpler criteria in order to determine the $c_{i,j}$-cyclic monotonicity of the $\Gamma_{i,j}$'s.   	
\end{remark}

\section{Classical cost functions and examples of $c$-cyclically monotone sets with and without $c_{i,j}$-cyclically monotone 2-marginal projections}

Let $H$ be a real Hilbert space. In the case where $X_i=H$ for
each $1\leq i\leq N$, a natural way to generalize the cost
function $\scal{\cdot}{\cdot}$, the inner product on $H$, from
the case $N=2$ to the case $N\geq 2$ is to consider
$c_{i,j}=\scal{\cdot}{\cdot}$ for each $1\leq i< j\leq N$
in~\eqref{our c}, that is, the function $c_1$ given
by~\eqref{classical c} below. An early study of multi-marginal
$c$-cyclic monotonicity for classical cost functions
is~\cite{KS}.  This was followed by an extensive study of
multi-marginal optimal transport for these costs in~\cite{GS}.
Similar to the situation in the two-marginal case, by now, the
classical cost functions are probably the most studied ones in
the multi-marginal optimal transport literature as well. 
Let $d\in\{3,4,\ldots\}$. In the case $X_i=\RR^d$ for each $1\leq
i\leq N$ and $N=d$, a natural cost function which is not of the
form~\eqref{our c} is $c(x_1,\ldots,x_N)=\det(x_1,\ldots,x_d),\
(x_1,\ldots,x_d)\in\RR^{d\times d}$ which was studied
in~\cite{CN}. However, we now focus our discussion on the cost functions $c_1,c_2$ and $c_3$ in the following definition. 

\begin{definition}\label{classic costs}
	For each $1\leq i\leq N$, set $X_i=H$. We let $c_1:X\to\RR$ be the function
	\begin{equation}\label{classical c}
	c_1(x_1,\dots,x_N)=\sum_{1\leq i<j\leq N}\scal{x_i}{x_j},
	\end{equation}
	we let $c_2:X\to\RR$ be the function
	\begin{equation}
	c_2(x_1,\dots,x_N)=\sum_{1\leq i<j\leq N}\tfrac{1}{2}\|x_i-x_j\|^2,
	\end{equation}
	and, finally, we let $c_3:X\to\RR$ be the function
	\begin{equation}
	c_3(x_1,\dots,x_N)=\tfrac{1}{2}\bigg\|\sum_{i=1}^N x_i\bigg\|^2.
	\end{equation}
\end{definition}

In the case $N=2$, the notion of $n-c_i$-monotonicity for $1\leq i\leq 3$ is the classical notion of $n$-monotonicity. In this case we omit $c$ from the notion and simply say ``monotone'', or ``$n$-monotone''. Two elementary and known (see, for example,~\cite{KS}) properties of $c_1,c_2$ and $c_3$ we shall employ are:

\begin{fact}\label{c_i equivalent}
	Let $\Gamma$ be a subset of $X$, 
	and let $n\in\{2,3,\dots\}$. Then the following assertions are equivalent:
	\begin{enumerate}
		\item $\Gamma$ is $n$-$c_1$-monotone.
		\item $\Gamma$ is $n$-$(-c_2)$-monotone.
		\item $\Gamma$ is $n$-$c_3$-monotone.
	\end{enumerate}
\end{fact}

\begin{proof}
(i) $\Leftrightarrow$ (ii) follows from the fact that 
$
c_2(x_1,\dots,x_N)=-c_1(x_1,\dots,x_N)+(N-1)\sum_{1\leq i\leq N}q(x_i)
$ for all $(x_1,\dots,x_N)\in X$
and by letting $h_i=(N-1)q$ in Remark~\ref{trivial monotonicity}. 
Similarly, (i) $\Leftrightarrow$ (iii) follows from the fact that
$
c_3(x_1,\dots,x_N)=c_1(x_1,\dots,x_N)+\sum_{1\leq i\leq N}q(x_i)
$
for all $(x_1,\dots,x_N)\in X$ and by letting $h_i=q$ in Remark~\ref{trivial monotonicity}.
\end{proof}

\begin{fact}\label{mono under shift}
Let $c\in\{c_1,c_2,c_3\}$ and let $z=(z_1,\ldots,z_N)\in
X$. If the subset $\Gamma$ of $X$ is $n$-$c$-cyclically monotone,
then so is $\Gamma+z$.
\end{fact}

\begin{proof}
In view of Fact~\ref{c_i equivalent}, we may assume, without the loss of generality, that $c=c_3$. We set $w=\sum_{1\leq i\leq N} z_i$. Then $c_3(x_1+z_1,\dots,x_N+z_N)=c_3(x_1,\dots,x_N)+\sum_{1\leq i\leq N}\scal{x_i}{w}+q(w)$. Consequently, the proof follows by letting $h_i=\scal{\cdot}{w}+\frac{q}{N}$ in Remark~\ref{trivial monotonicity}.
\end{proof}	

We now present two examples. The first example demonstrates the
advantages of our approach, such as described in Remark~\ref{easier
condition}, in identifying particular $c$-cyclically monotone sets
which are the subject of matter and in explicitly computing
$c$-splitting tuples.

\begin{example}
Let $Q_1\in\RR^{d\times d}$ and $Q_2\in\RR^{d\times d}$ be 
symmetric and positive definite. We recall that if $Q_1$
and $Q_2$ commute, then $Q_1Q_2$ is symmetric and positive definite
(see, for example,~\cite[Theorem~4.6.9]{DM} or~\cite[Chapter~VII]{RS}).
Furthermore, in this case, since $Q_1$ and $Q_2^{-1}$ also commute,
then $Q_1Q_2^{-1}$ is symmetric and positive definite. These facts
give rise to the following example: For each $1\leq i\leq N$, set
$X_i=\RR^d$ and let $Q_i\in\RR^{d\times d}$ be symmetric,
positive definite, and pairwise commuting. 
Set 
$$
\Gamma=\big\{(Q_1v,\ldots,Q_Nv)\ \big|\ v\in\RR^d\big\}.
$$
Then, for each $1\leq i<j\leq N$, we have
$$
\Gamma_{i,j}=\big\{(Q_i v,Q_j v)\ \big|\ v\in\RR^d\big\}=
\big\{(w,Q_jQ_i^{-1}w)\ \big|\ w\in\RR^d\big\}\ \ \ \ \ \ \text{and}\ \ \ \ \ \ \ \Gamma=\bigcap_{i<j}P_{i,j}^{-1}(\Gamma_{i,j}).
$$
Since $Q_jQ_i^{-1}$ is symmetric and positive definite, we see that
$\Gamma_{i,j}$ is monotone. We now explain how 
Theorem~\ref{main abstract result} is applicable. 
To this end, we set
$f_{i,j}=q_{Q_jQ_i^{-1}}$.
Then
$$
\partial f_{i,j}=\nabla f_{i,j}=P_jP_i^{-1}\ \ \ \ \ \ \text{and}\ \ \ \ \ \ \ \gra(\partial f_{i,j})=\Gamma_{i,j}.
$$  
Furthermore,
$$
f_{i,j}^*=q^*_{Q_jQ_i^{-1}}=q_{(Q_jQ_i^{-1})^{-1}}=q_{Q_iQ_j^{-1}}=f_{j,i}.
$$
We also set
$$
u_i=\sum_{i<k\leq N}f_{i,k}+\sum_{1\leq k<i}f_{k,i}^*=\sum_{k\neq
i}f_{i,k}=\sum_{k\neq i}q_{Q_kQ^{-1}_i}
=q_{\sum_{k\neq i}Q_kQ^{-1}_i}=q_{M_i}.
$$
where $M_i\in\RR^{d\times d}$ is defined by 
$$
M_i=\Big(\sum_{k\neq i}Q_k\Big)Q_i^{-1}.
$$
Finally, since \eqref{main furthermore} holds, by applying Theorem~\ref{main abstract result}, we see that
$$
c_1(x_1,\ldots,x_N)=\sum_{1\leq i< j\leq N}\scal{x_i}{x_j}\leq\sum_{1\leq i\leq N}q_{M_i}(x_i)\ \ \ \ \ \ \ \ \ \ \ \text{for all}\ \ (x_1,\ldots,x_N)\in X
$$
and equality holds if and only if $(x_1,\ldots,x_N)\in\Gamma$. If we set $G_i=\Id+M_i=\Big(\sum_{1\leq k\leq N}Q_k\Big)Q_i^{-1}$, we can write, equivalently, 
$$
c_3(x_1,\ldots,x_N)=\bigg\|\sum_{1\leq i\leq N}x_i\bigg\|^2\leq\sum_{1\leq i\leq N}q_{G_i}(x_i)\ \ \ \ \ \ \ \ \ \ \ \text{for all}\ \ (x_1,\ldots,x_N)\in X
$$
and equality holds if and only if $(x_1,\ldots,x_N)\in\Gamma$.
\end{example}

In the following example we demonstrate that the $\Gamma_{i,j}$'s
being $c_{i,j}$-cyclically monotone is a sufficient condition,
however, it is not necessary in order that $\Gamma$ be $c$-cyclically monotone and a $c$-splitting set.

\begin{example}
Suppose that $N=3$ and that $X_1 = X_2=X_3=\RR^2$.  We set
$$
A_1=2\begin{pmatrix}
1 & 0\\0 & 0
\end{pmatrix},\ \ \ \
A_2=2	\begin{pmatrix}
1 & 0\\0 & 1
\end{pmatrix},\ \ \ \ 
A_3=\frac{1}{7}	\begin{pmatrix}
8 & 3\\3 & 2
\end{pmatrix}
$$
and
$$
\Delta=\big\{(a,a)\ \big|\ a\in\RR\big\} \subseteq\RR^2.
$$
%For each $1\leq i\leq N$, we employ the above notations in order to define  the proper, convex and lower semicontinuous function $u_i:\RR^2\to\RX$ by
Furthermore, set 
$$
u_1=\iota_{\RR\times\{0\}}+q_{A_1},\ \ \ \ \ \ \ u_2=\iota_\Delta+q_{A_2}=\iota_\Delta+2q,\ \ \ \ \ \ \ \ \ \text{and}\ \ \ \ \ \ \ \ \ u_3=q_{A_3}.
$$
Our aim is to study the $c$-cyclic monotonicity properties of the 
subset $\Gamma$ of $X$ defined by:
\begin{equation}\label{non mono shadows gamma}
\Gamma=\Big\{(x_1,x_2,x_3)\in\big(\RR^2\big)^3\ \Big|\ \scal{x_1}{x_2}+\scal{x_2}{x_3}+\scal{x_3}{x_1}= u_1(x_1)+u_2(x_2)+u_3(x_3)\ \Big\}.
\end{equation}
To this end, we claim the following:
\begin{enumerate}
	\item $\scal{x_1}{x_2}+\scal{x_2}{x_3}+\scal{x_3}{x_1}\leq u_1(x_1)+u_2(x_2)+u_3(x_3)\ \ \ $ for all $(x_1,x_2,x_3)\in\big(\RR^2\big)^3$;\\
	
	\item Let $v_1=\big((0,0),(-1,-1),(1,-5)\big)$ and $v_2=\big((1,0),(2,2),(0,7)\big)$. Then $\Gamma=\spa\{v_1,v_2\}$;\\
	
	\item $\Gamma_{1,2}, \Gamma_{1,3}$ and $\Gamma_{2,3}$ are not monotone.
\end{enumerate}
Before we prove these claims, let us discuss their consequences:
By combining~\eqref{non mono shadows gamma} with (i) we see that
$(u_1,u_2,u_3)$ is a $c_1$-splitting tuple of $\Gamma$.
Consequently, Fact~\ref{splitting implies cyclic} now implies
that $\Gamma$ is $c_1$-cyclically monotone. On the other hand,
(iii) implies that the $\Gamma_{i,j}$'s are not monotone. In
summary,
$\Gamma$ is a $c$-cyclically monotone set (with the explicit
splitting tuple $(u_1,u_2,u_3)$) such that all of its 
2-marginal projections $\Gamma_{1,2}$, $\Gamma_{1,3}$ and $\Gamma_{2,3}$
are nonmonotone. We therefore conclude that the condition requiring the $\Gamma_{i,j}$'s to be $c_{i,j}$-cyclically monotone for all $1\leq i<j\leq N$ is only a sufficient condition implying that $\Gamma$ is a $c$-splitting (and $c$-cyclically monotone) set; however, it is not a necessary condition. 
\end{example}

We now turn to proving (i)--(iii):

\begin{proof}
We set $x_1=(a_1,b_1),\ x_2=(a_2,b_2)$ and $x_3=(a_3,b_3)$,
 and we will prove that
\begin{equation}\label{u inequality}
u_1(x_1)+u_2(x_2)+u_3(x_3)-\scal{x_1}{x_2}-\scal{x_2}{x_3}-\scal{x_3}{x_1}
\geq 0. 
\end{equation}
Since $u_1(a_1,b_1)=\infty$ whenever $b_1\neq 0$ and since
$u_2(a_2,b_2)=\infty$ whenever $a_2\neq b_2$, it is enough to
prove~\eqref{u inequality} in the case $b_1=0$ and $b_2=a_2$,
which we assume from now on.
Then
\begin{align*}
&u_1(x_1)+u_2(x_2)+u_3(x_3)-\scal{x_1}{x_2}-\scal{x_2}{x_3}-\scal{x_3}{x_1}\\
=\ &u_1(a_1,0)+u_2(a_2,a_2)+u_3(a_3,b_3)-\scal{(a_1,0)}{(a_2,a_2)}-\scal{(a_2,a_2)}{(a_3,b_3)}-\scal{(a_3,b_3)}{(a_1,0)}\\
=\ &a_1^2+2a_2^2+\tfrac{1}{7}(4a_3^2+3a_3b_3+b_3^2)-a_1a_2-a_2a_3-a_2b_3-a_1a_3=\scal{x}{Mx}=\scal{x}{\sym(M)x}
\end{align*}
where $x\in\RR^4$ is given by $x=(a_1,a_2,a_3,b_3)$, $M\in\RR^{4\times 4}$ is the matrix given by
$$
M=
\begin{pmatrix}
1 & -1 & -1 & 0\\[11pt]
0 & 2 & -1 & -1\\[11pt]
0 & 0 & \dfrac{4}{7} & \dfrac{3}{7}\\[11pt]
0 & 0 & 0 & \dfrac{1}{7}
\end{pmatrix}
\ \ \ \ \ \ \ \ \ \text{and}\ \ \ \ \
\sym(M)=\frac{1}{2}(M+M^T)=
\begin{pmatrix}
1 & -\dfrac{1}{2} & -\dfrac{1}{2} & 0\\[11pt]
-\dfrac{1}{2} & 2 & -\dfrac{1}{2} & -\dfrac{1}{2}\\[11pt]
-\dfrac{1}{2} & -\dfrac{1}{2} & \dfrac{4}{7} & \dfrac{3}{14}\\[11pt]
0 & -\dfrac{1}{2} & \dfrac{3}{14} & \dfrac{1}{7}
\end{pmatrix}.
$$
The characteristic polynomial of $\sym(M)$ is
$$
\frac{1}{28}z^2(z^2-104z+89)=z^2\Big(z-\frac{26-\sqrt{53}}{14}\Big)\Big(z-\frac{26+\sqrt{53}}{14}\Big).
$$
We see that the eigenvalues of $\sym(M)$ are nonnegative. 
Consequently, $\sym(M)$ is positive semidefinite, which completes the proof of (i). Furthermore,
$$
\ker\big(\sym(M)\big)=\spa\big\{(0,-1,1,-5),(1,2,0,7)\big\}.
$$
By recalling that $b_1=0$ and $a_2=b_2$ we arrive at (ii). Thus, we now see that $(0,0,0)\in\Gamma$ and that for any $\lambda\in\RR$,
$$
\big((x_1(\lambda),x_2(\lambda),x_3(\lambda)\big)=\big((1,0),(2-\lambda,2-\lambda),(\lambda,7-5\lambda)\big)=\lambda v_1+v_2\in\Gamma.
$$
Consequently, 
\begin{align}
&\scal{x_1(\lambda)-0}{x_2(\lambda)-0}=2-\lambda<0\ \ \ \ \ &\Leftrightarrow\ \ \ \ \ \ &2<\lambda;\label{12 non-mono}\\
\nonumber\\
&\scal{x_1(\lambda)-0}{x_3(\lambda)-0}=\lambda<0\ \ \ \ \ &\Leftrightarrow\ \ \ \ \ \ &\lambda<0;\label{13 non-mono}\\
\nonumber\\
&\scal{x_2(\lambda)-0}{x_3(\lambda)-0}=(2-\lambda)(7-4\lambda)<0\ \ \ \ \ &\Leftrightarrow\ \ \ \ \ \ &\frac{7}{4}<\lambda<2.\label{23 non-mono}
\end{align}
\eqref{12 non-mono} implies that $\Gamma_{1,2}$ is not monotone, \eqref{13 non-mono} implies that $\Gamma_{1,3}$ is not monotone and, finally, \eqref{23 non-mono} implies that $\Gamma_{2,3}$ is not monotone which completes the proof.
\end{proof}

Our discussion thus far raises the following natural, currently
unsolved, questions: 
\ding{192}
For which cost functions $c$ of the form~\eqref{our c},
$c$-cyclic monotonicity of a set $\Gamma$ implies the
$c_{i,j}$-cyclic monotonicity of its $\Gamma_{i,j}$'s?
\ding{193}
Given a cost function $c$, for which sets $\Gamma$, $c$-cyclic monotonicity of a set $\Gamma$ implies the $c_{i,j}$-cyclic monotonicity of its $\Gamma_{i,j}$'s?

\section{Multi-marginal classical cost functions in the one-dimensional case}

In this section, we focus our attention to the case where $X_i=\RR$ for each $1\leq i\leq N$ and $c:X\to\RR$ is given by $c=c_1$, that is,
\begin{equation}\label{1d c}
c(x_1,\ldots,x_N)=\sum_{1\leq i<j\leq N} x_ix_j.
\end{equation}
(Equivalently, we may consider $c=c_2$ or $c=c_3$.) For a more
general class of costs it was established in~\cite{Car} that if
$\Gamma\subseteq X$ is a $c$-splitting set then 2-marginal
projections $\Gamma_{i,j}$ are monotone in $\RR^2$. A more
elementary proof of this fact was provided in~\cite{Pas2}. 
We thank an anonymous referee for pointing out these connections.
For the sake of completeness of
our discussion and for the convenience of the reader we include
below a proof of this fact in the spirit of~\cite{Pas2} for the
cost $c$ in \eqref{1d c}.  We then combine this fact with our
discussion in the previous sections, and obtain characterizations
of $c$-splitting sets. 
Thus, our aim is to
show that $\Gamma$ being $c$-cyclically monotone
is, in fact, equivalent to the $\Gamma_{i,j}$'s being cyclically
monotone. Furthermore, for $N=2$, it is well known that $\Gamma$
is monotone if and only if it is cyclically monotone and that in
this case $\Gamma$ is a splitting set. We will conclude that these
elementary facts from classical convex function theory on the
real line hold in the multi-marginal case as well. To this end we
will make use of the following lemma which, geometrically,
asserts the following: In the case $N=2$, the set $\Gamma$ is monotone if and only if for any point $(z_1,z_2)\in\Gamma$, the translated set $\Gamma-(z_1,z_2)$ is contained in the first and third quarters of the plane, that is, in 
$\RP^2\cup\RM^2$,
where $\RP = \menge{x\in \RR}{x\geq 0}$
and $\RM=\menge{x\in\RR}{x\leq 0}$. 
The following lemma asserts that for any $N\geq 2$, the set $\Gamma$ is $c_1$-monotone if and only if for any point $(z_1,\ldots,z_N)\in\Gamma$ the set $\Gamma-(z_1,\ldots,z_N)$ is contained in 
$\RP^N\cup\RM^N$.
\begin{lemma}\label{real line mono signs}
Let $X_i=\RR$ for each $1\leq i\leq N$ and set $c=c_1$.
If $(t_1,\ldots,t_N)\in X$ is in $c$-monotone relations with $(0,\ldots,0)$ (that is, the set $\{(t_1,\ldots,t_N),\ (0,\ldots,0)\}$ is $c$-monotone), then all of the $t_i$'s have the same sign, that is, 
$(t_1,\ldots,t_N)\in\RP^N\cup\RM^N$.
\end{lemma}

\begin{proof}
We argue by contradiction. Thus, we assume to the contrary that
$(t_1,\ldots,t_N)$ is $c$-monotonically related to $(0,\ldots,0)$
and that not all of the $t_i$'s have the same sign. We define a
partition of the index set $I=\{1,\ldots,N\}$ by $I_+=\{i\in I\ |\
t_i\geq 0 \}$
and $I_-=\{i\in I\ |\ t_i<0\ \}$. Consequently, 
$$
\bigg(\sum_{i\in I_+}t_i\bigg)\bigg(\sum_{i\in I_-}t_i\bigg)<0.
$$ 
For each $1\leq i\leq N$, we define 
$$
t_i^+=\begin{cases}
t_i, &\text{if $i\in I_+$;}\\ 
0, &\text{if $i\in I_-$}
\end{cases} 
\qquad \text{and}\qquad
t_i^-=\begin{cases}
0, & \text{if $i\in I_+$;} \\ 
t_i & \text{if $i\in I_-$.}
\end{cases}
$$
Finally, by employing the definition of $c$-monotonicity 
and our notation above we arrive ---after some algebraic
manipulations---at
\begin{align*}
0&\leq c(t_1,\ldots,t_N)+c(0,\ldots,0)-c(t_1^+,\ldots,t_N^+)-c(t_1^-,\ldots,t_N^-)\\
\\
&=\sum_{i,j\in I,\ i<j} t_i t_j\ \ \ \ \ \ +0\ \ \ \ \ \ -\sum_{i,j\in I,\ i<j} t_i^+ t_j^+-\sum_{i,j\in I,\ i<j} t_i^- t_j^-\\
\\
&=\sum_{i,j\in I,\ i<j} t_i t_j-\sum_{i,j\in I^+,\ i<j} t_i
t_j-\sum_{i,j\in I^-,\ i<j} t_i t_j=\Big(\sum_{i\in
I_+}t_i\Big)\Big(\sum_{i\in I_-}t_i\Big)<0,
\end{align*}
which is the desired contradiction.
\end{proof}

\begin{theorem}\label{mono on the real line}
Let $X_i=\RR$ for each $1\leq i\leq N$, and set $c=c_1$ (equivalently, $c=-c_2$ or $c=c_3$).
For a subset $\Gamma$ of $X=\RR^N$, the following assertions are equivalent:
\begin{enumerate}
\item $\Gamma$ is $c$-cyclically monotone in $\RR^N$.
\item $\Gamma$ is $c$-monotone in $\RR^N$.
\item $\Gamma_{i,j}$ is cyclically monotone in $\RR^2$ for each $1\leq i< j\leq N$.
\item $\Gamma_{i,j}$ is monotone in $\RR^2$ for each $1\leq i< j\leq N$.
\item For each $1\leq i\leq N$, there exist a proper, convex and lower semicontinuous function $u_i:\RR\to\RX$ such that $(u_1,\ldots,u_N)$ is a $c$-splitting tuple of $\Gamma$.
\item For each $1\leq i< j\leq N$ there exist a proper, convex and lower semicontinuous function $f_{i,j}:\RR\to\RX$ such that $\Gamma_{i,j}\subseteq\gra(\partial f_{i,j})$.
\end{enumerate}
In this case, for each $1\leq i\leq N$, one can take 
\begin{equation}\label{u real line}
 u_i(x_i)=\sum_{i<k}f_{i,k}(x_i)+\sum_{k<i}f_{k,i}^*(x_i).
\end{equation}
\end{theorem}

\begin{proof}
The equivalence (iii)~$\Leftrightarrow$~(iv) of monotonicity and
cyclic monotonicity in the two-marginal case on the real line is
well known (see, for example, \cite[Theorem 22.18]{BC2011}). The
equivalence (iii)~$\Leftrightarrow$~(vi) follows from
Fact~\ref{Rock}. The implication (v)~$\Rightarrow$~(i) is a
consequence of Fact~\ref{splitting implies cyclic}. The
implications (iii)~$\Rightarrow$~(i) and (iii)~$\Rightarrow$~(v)
via (vi) combined with~\eqref{u real line} is a consequence of
Theorem~\ref{main abstract result}. (i)~$\Rightarrow$~(ii) is
trivial. Thus, in order to complete the proof it is enough to
prove the implication (ii)~$\Rightarrow$~(iv). To this end let
$1\leq k<j\leq N$ and let $(x_k,x_j),(y_k,y_j)\in\Gamma_{kj}$. We
need to prove that $(x_k-y_k)(x_j-y_j)\geq 0$. Since $(x_k,x_j),(y_k,y_j)\in\Gamma_{kj}$, there exist $x=(x_1,\ldots,x_k,\ldots,x_j,\ldots,x_N)\in\Gamma$ and $y=(y_1,\ldots,y_k,\ldots,y_j,\ldots,y_N)\in\Gamma$. By combining our assumption that $\Gamma$ is $c$-monotone with Proposition~\ref{mono under shift} we see that $\Gamma-y$ is $c$-monotone, that is, $x-y$ is $c$-monotonically related to $(0,\ldots,0)$. Finally, we invoke Lemma~\ref{real line mono signs} with $t_i=x_i-y_i$ for each $1\leq i\leq N$ in order to conclude that $t_k$ and $t_j$ have the same sign, that is, 
$(x_k-y_k)(x_j-y_j) = t_k t_j \geq 0$.
\end{proof}

In the following example we discuss the class of all (according to Theorem~\ref{mono on the real line}) $c_1$-monotone (continuous with onto projections on the axis) curves in $\RR^N$. Our discussion can be generalized; however, for clearness of our presentation, we impose a continuity assumption.

\begin{example}
For each $1\leq i\leq N$ let $\alpha_i:\RR\to\RR$ be a continuous, strictly increasing and onto function with $\alpha_i(0)=0$. We consider the curve $\Gamma$ in $\RR^N$ defined by
$$
\Gamma=\Big\{\big(\alpha_1(t),\ldots,\alpha_N(t)\big)\ \Big|\ t\in\RR\Big\}.
$$
Then for each $1\leq i< j\leq N$ the set $\Gamma_{i,j}=\Big\{\big(\alpha_i(t),\alpha_j(t)\big)\ |\ t\in\RR \Big\}$ is clearly a monotone set and $\Gamma=\bigcap_{i<j}P_{i,j}^{-1}(\Gamma_{i,j})$. Consequently, $\Gamma$ is a $c$-monotone set where $c=c_1$. We define $f_{i,j}:\RR\to\RR$ by
\begin{equation}\label{curves f}
f_{i,j}(x_i)=\int_0^{x_i}\alpha_j\big(\alpha^{-1}_i(t)\big)dt.
\end{equation}
Then $f_{i,j}$ is convex, differentiable and
$$
\gra(\partial f_{i,j})=\gra(f')=\Big\{\big(x_i,\alpha_j(\alpha^{-1}_i(x_i)\big)\ \Big|\ x_i\in\RR\Big\}=\Gamma_{i,j}.
$$
Furthermore,
\begin{equation}\label{curves f^*}
f_{i,j}^*(x_j)=\int_0^{x_j}\alpha_i\big(\alpha^{-1}_j(t)\big)dt =
f_{j,i}(x_j).
\end{equation}
Thus, after plugging~\eqref{curves f} and~\eqref{curves f^*} into~\eqref{u real line}, we arrive at
\begin{equation}\label{curves}
u_i(x_i)=\int_0^{x_i} \bigg(\sum_{k\neq i} \alpha_k\big(\alpha^{-1}_i(t)\big)\bigg)dt.
\end{equation}
In summary, since \eqref{main furthermore} holds, by recalling Definition~\ref{splitting def}, Theorem~\ref{main abstract result} and
Theorem~\ref{mono on the real line}, 
we conclude that 
\begin{equation}\label{curve inequality}
\sum_{1\leq i<j\leq N} x_ix_j \leq \sum_{i=1}^N\int_0^{x_i} \bigg(\sum_{k\neq i} \alpha_k\big(\alpha^{-1}_i(t)\big)\bigg)dt
\ \ \ \ \ \ \ \ \ \ \ \ \ \forall (x_1,\ldots,x_N)\in\RR^N
\end{equation}
and that equality in~\eqref{curve inequality} holds if and only if $x_j=\alpha_j\big(\alpha_i^{-1}(x_i)\big)$ for every $1\leq i<j\leq N$.
In the case $N=2$, we set $g=\alpha_2\circ\alpha_1^{-1}$, $a=x_1$ and $b=x_2$. Then the latter reduces to the well-known version of Young's inequality
$$
ab\leq\int_0^ag(t)dt+\int_0^b g^{-1}(t)dt\ \ \ \ \ \ \ \ \ \ \ \ \ \ \ \ \ \forall a,b
$$
with equality if and only if $b=g(a)$.
\end{example}

We conclude with a demonstration of the computational advantages of our approach by elaborating on one of the earliest examples in the literature.

\begin{example}
\label{ex:KS}
In~\cite{KS}, Knott and Smith considered the setting: Set $N=3$, 
$X_i=\RR$ for $1\leq i\leq 3$, and $\Gamma=\big\{(t,t^3,t^5)\ |\
t\in\RR\big\}$. Then $t=\phi(w)$ was defined to be the inverse function of $t+t^3+t^5=w$ and also the following functions were defined
\begin{align*}
\alpha(v)&=\int_0^v\phi(w)dw,\\
\\
\beta(v)&=\int_0^v\phi^3(w)dw,\\
\\
\gamma(v)&=\int_0^v\phi^5(w)dw.
\end{align*}
It was then concluded that 
\begin{equation}\label{Knott Smith inq}
\tfrac{1}{2}|x_1+x_2+x_3|^2\leq\alpha^* (x_1)+\beta^*
(x_2)+\gamma^* (x_3),
\end{equation}
with equality for $(x_1,x_2,x_3)\in\Gamma$. We now address this example using our approach and construct $\alpha^*,\beta^*,\gamma^*$ explicitly. Using our notation, we set $\alpha_1(t)=t,\ \alpha_2(t)=t^3$ and $\alpha_3(t)=t^5$. Then by combining $\alpha_1^{-1}(t)=t,\ \alpha_2^{-1}(t)=t^{\frac{1}{3}},\ \alpha_3^{-1}(t)=t^{\frac{1}{5}}$ with~\eqref{curves} we conclude that the functions
\begin{align*}
u_1(x_1)&=\int_0^{x_1}\Big(\alpha_2\big(\alpha_1^{-1}(t)\big)+\alpha_3\big(\alpha_1^{-1}(t)\big)\Big)dt=
\int_0^{x_1}\big(t^3+t^5\big)dt
=\tfrac{1}{4}{x_1^4}+\tfrac{1}{6}{x_1^6},\\
\\
u_2(x_2)&=\int_0^{x_2}\Big(\alpha_1\big(\alpha_2^{-1}(t)\big)+\alpha_3\big(\alpha_2^{-1}(t)\big)\Big)dt
=\int_0^{x_2}\big(t^{\frac{1}{3}}+t^{\frac{5}{3}}\big)dt
=\tfrac{3}{4}x_2^{{4}/{3}}+\tfrac{3}{8}x_2^{{8}/{3}},\\
\\
u_3(x_3)&=\int_0^{x_3}\Big(\alpha_1\big(\alpha_3^{-1}(t)\big)+\alpha_2\big(\alpha_3^{-1}(t)\big)\Big)dt
=\int_0^{x_3}\big(t^{\frac{1}{5}}+t^{\frac{3}{5}}\big)dt
=\tfrac{5}{6}x_3^{{6}/{5}}+\tfrac{5}{8}x_3^{{8}/{5}}
\end{align*}
satisfy 
$$
x_1x_2+x_2x_3+x_3x_1\leq u_1(x_1)+u_2(x_2)+u_3(x_3)\ \ \ \ \ \ \
\forall(x_1,x_2,x_3)\in\RR^3,
$$
with equality if and only if $(x_1,x_2,x_3)\in\Gamma$. It follows that
$$
\tfrac{1}{2}|x_1+x_2+x_3|^2\leq
(u_1+q)(x_1)+(u_2+q)(x_2)+(u_3+q)(x_3)\ \ \ \ \ \ \
\forall(x_1,x_2,x_3)\in\RR^3, 
$$
with equality if and only if $(x_1,x_2,x_3)\in\Gamma$. Finally, relating our discussion back to the discussion of Knott and Smith, it is not hard to verify that $u_1+q=\alpha^*,\ u_2+q=\beta^*$ and $u_3+q=\gamma^*$. Furthermore, the case of equality in~\eqref{Knott Smith inq} is now characterized. 
Finally, in Figure~\ref{fig:KS},
we depict $\Gamma$ and its three planar projections.
\begin{figure}[h!]
\begin{center}
\includegraphics[width=0.7\columnwidth]{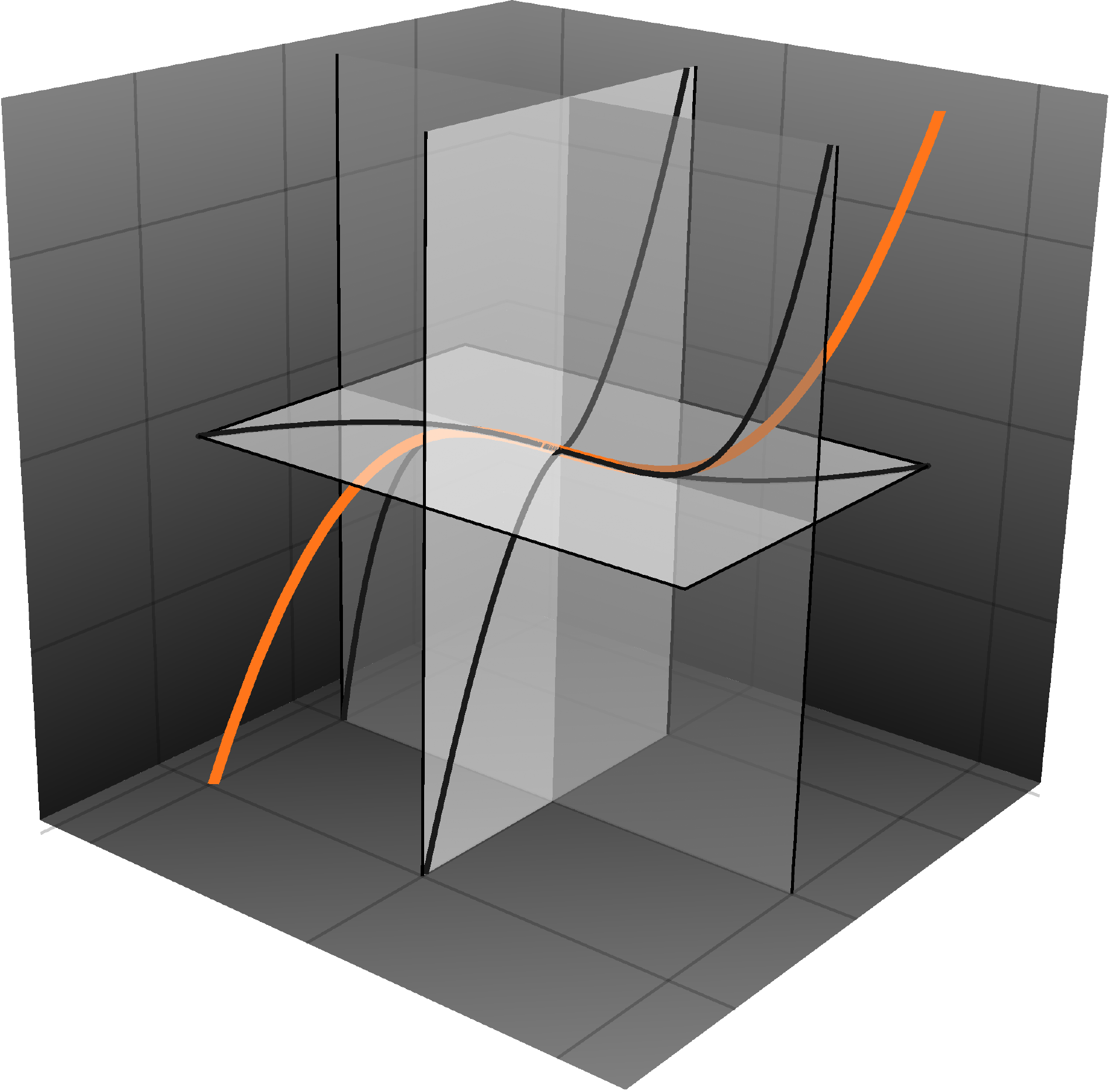}
\caption{The curve $\Gamma$ of Example~\ref{ex:KS} together
with its planar projections $\Gamma_{1,2}$, $\Gamma_{2,3}$, and
$\Gamma_{1,3}$.}
 \label{fig:KS}
\end{center}
\end{figure}

\end{example}

\newpage

\section*{Acknowledgments}
Sedi Bartz was supported by a postdoctoral fellowship of the Pacific
Institute for the Mathematical Sciences and by NSERC grants of Heinz
Bauschke and Xianfu Wang. Heinz Bauschke was partially supported
by the Canada Research Chair program and by the Natural Sciences
and Engineering Research Council of Canada. 
Xianfu Wang was partially supported by the Natural
Sciences and Engineering Research Council of Canada.

%\footnotesize

\end{document}